\documentclass[oneside,english]{amsart}
\usepackage[T1]{fontenc}
\usepackage[latin9]{inputenc}
\usepackage{amsthm}
\usepackage{amstext}
\usepackage{amssymb}
\usepackage{esint}

\makeatletter

\DeclareFontEncoding{LGR}{}{}
\DeclareTextSymbol{\~}{LGR}{126}
\newcommand{\lyxmathsym}[1]{\ifmmode\begingroup\def\b@ld{bold}
  \text{\ifx\math@version\b@ld\bfseries\fi#1}\endgroup\else#1\fi}

\numberwithin{equation}{section}
\numberwithin{figure}{section}
\theoremstyle{plain}
\newtheorem{thm}{\protect\theoremname}[section]
  \theoremstyle{plain}
  \newtheorem{lem}[thm]{\protect\lemmaname}
  \theoremstyle{definition}
  \newtheorem{defn}[thm]{\protect\definitionname}
  \theoremstyle{plain}
  \newtheorem{prop}[thm]{\protect\propositionname}
  \theoremstyle{plain}
  
  \theoremstyle{definition}

\makeatother

\usepackage{babel}
  \providecommand{\corollaryname}{Corollary}
  \providecommand{\definitionname}{Definition}
  \providecommand{\examplename}{Example}
  \providecommand{\lemmaname}{Lemma}
  \providecommand{\propositionname}{Proposition}
\providecommand{\theoremname}{Theorem}

\begin{document}

\centerline{\bf Conformal Spectral Stability Estimates for the Dirichlet Laplacian}

\vskip 0.3cm
\centerline{\bf V.~I.~ Burenkov,*\, V.~Gol'dshtein, A.~Ukhlov**}
\vskip 0.3cm
{\small
{\it Address and affiliations:}

V.~I.~ Burenkov: Professor,
Peoples' Friendship University of Russia: Russia, Moscow, 5 Mikluho-Maklay St., Steklov Mathematical Institute: Russia, Moskow, 8 Gubkin St.
E-mail: Burenkov@cardiff.ac.uk

V.~Gol'dshtein: Professor,
Ben-Gurion University of the Negev: Israel, 84105, Beer-Sheva, P.O. Box 653.
E-mail: vladimir@math.bgu.ac.il

A.~Ukhlov: Associate Professor,
Ben-Gurion University of the Negev: Israel, 84105, Beer-Sheva, P.O. Box 653.
972-8-6477830. E-mail: ukhlov@math.bgu.ac.il
}

\vskip 0.3cm

\centerline{\small ABSTRACT}
\vskip 0.1cm
{\small 
We study the eigenvalue problem for the Dirichlet Laplacian in bounded simply connected plane domains $\Omega\subset\mathbb{C}$
using conformal transformations of the original problem to the weighted eigenvalue problem for the Dirichlet Laplacian
in the unit disc $\mathbb{D}$. This allows us to estimate the variation of the eigenvalues of the Dirichlet Laplacian upon domain perturbation via energy type integrals for a large class of "conformal regular" domains which includes all quasidiscs, i.e. images of the unit disc under  quasiconformal homeomorphisms of the plane onto itself. Boundaries of such domains can have any Hausdorff dimension between one and two.
}
\vskip 0.1cm




\footnotetext{* The author was partially supported by the Center for Advanced Studies in Mathematics at the Ben-Gurion University of the Negev and by the Russian Scientific Foundation (project 14-11-00443).}
\footnotetext{** The corresponding author.}

\section{Introduction }

This paper is devoted to stability estimates for the  eigenvalues of the Dirichlet Laplacian
$$
-\Delta f = -\Big(\frac{\partial^2 f}{\partial x^2}+\frac{\partial^2 f}{\partial y^2}\Big),\,\,\, (x,y)\in \Omega,\,\,\,\,\, f\vert_{\partial\Omega}=0.
$$

It is known that in a bounded plane domain $\Omega\subset\mathbb{C}$ the spectrum of the Dirichlet Laplacian is discrete and can be written in the form
of a non-decreasing sequence
\[
0<\lambda_{1}[\Omega]\leq\lambda_{2}[\Omega]\leq...\leq\lambda_{n}[\Omega]\leq...\,,
\]
where each eigenvalue is repeated as many times as its multiplicity.

In the last two decades, spectral stability estimates for the Dirichlet Laplacian were intensively studied. See, for example, \cite{P, D, LL, BLL, BL1, BL2, LP, BBL, BL3},
where the quantity
$|\lambda_{n}[\Omega_{1}]-\lambda_{n}[\Omega_{2}]|$, under certain assumptions on regularity of the domains $\Omega_1$ and $\Omega_2$,
was estimated via various characteristics of the closeness of $\Omega_1$ and $\Omega_2$ such as the so-called atlas distance between
$\Omega_1$ and $\Omega_2$, the Hausdorff-Pompeiu distance between the boundaries $\partial\Omega_1$ and $\partial\Omega_2$,
the Lebesgue measure of the symmetric difference of $\Omega_1$ and $\Omega_2$.

If $\varphi_1$ and $\varphi_2$ are Lipschitz mappings such that $\Omega_1=\varphi_1(\mathbb D)$ and $\Omega_2=\varphi_2(\mathbb D)$, where $\mathbb D\subset\mathbb C$ is the unit disc, the dependence of $|\lambda_{n}[\Omega_{1}]-\lambda_{n}[\Omega_{2}]|$ on the closeness of the mappings $\varphi_1$ and $\varphi_2$ was investigated in \cite{LL}. See also \cite{BL1, BL2} and survey paper \cite{BLL}, where one can find references to other related results.

Let, for $\tau>0$, $F_\tau$ be the set of all mappings $\varphi$ of the unit disc $\mathbb D$ in the Sobolev class $L^{1,\infty}(\mathbb D)$ such that
$$
\|\varphi\mid {L^{1,\infty}(\mathbb D)}\|\le\tau\,,\,\,\,\,
{\begin{array}{c}\\{\rm ess~inf}\\{\mathbb D}\end{array}}\,|{\rm det}\,\nabla\varphi|\ge\frac1\tau\,.
$$
\begin{thm}\cite{LL}
For any $\tau>0$ there exists $A_\tau>0$ such that for any $\varphi_1,\varphi_2\in F_\tau$ and for any $n\in\mathbb N$
\begin{equation}\label{estimate 1}
|\lambda_{n}[\Omega_{1}]-\lambda_{n}[\Omega_{2}]|\leq c_n A_\tau \|\varphi_1-\varphi_2\mid L^{1,\infty}(\mathbb D)\|\,,
\end{equation}
where $\Omega_{1}=\varphi_1(\mathbb D)$, $\Omega_{2}=\varphi_2(\mathbb D)$ and
\begin{equation}\label{c_n}
c_{n}=\max\{\lambda_{n}^{2}[\Omega_{1}],\lambda_{n}^{2}[\Omega_{2}]\}\,.
\end{equation}
\end{thm}
This theorem also holds if $\mathbb D$ is replaced by any open set $\Omega\subset\mathbb{R}^N, N\ge 2,$ such that the embedding $W_{0}^{1,2}(\Omega)\hookrightarrow L^{2}(\Omega)$ is compact \cite{LL}. In this case $A_\tau $ depends also on the Poincar\'e constant of $\Omega$.

In \cite{BrL1} (Theorem~6) the stability estimates based on summability assumptions on the gradients of the eigenfunctions were obtained. 

In this paper we consider {\it conformal regular} plane domains $\Omega\subset\mathbb{C}$. We call a bounded simply connected plane domain $\Omega\subset\mathbb{C}$  a conformal regular domain if there exists a conformal mapping $\varphi:\mathbb D\to\Omega$ in the Sobolev class $L^{1,p}(\mathbb D)$ for some $p>2$. Note that any conformal regular domain has finite geodesic diameter 
\cite{GU5} and can be characterized in the terms of the (quasi)\-hy\-per\-bolic boundary condition \cite{BP, KOT}. 
For such domains we improve estimate (\ref{estimate 1}).

Let, for $2<p\le\infty, \tau>0$, $G_{p,\tau}$ be the set of all conformal mappings $\varphi$ of the unit disc $\mathbb D$ of the Sobolev class $L^{1,p}(\mathbb D)$ such that
$$
\|\varphi\mid {L^{1,p}(\mathbb D)}\|\le\tau\,.
$$

The main result of this paper is

\begin{thm} For any $2<p\le\infty, \tau>0$ there exists $B_{p,\tau}>0$ such that for any $\varphi_1,\varphi_2\in G_{p,\tau}$ and for any $n\in\mathbb N$
\begin{equation}\label{estimate 2}
|\lambda_{n}[\Omega_{1}]-\lambda_{n}[\Omega_{2}]|
\leq c_n B_{p,\tau} \|\varphi_1-\varphi_2\mid L^{1,2}(\mathbb D)\|\,,
\end{equation}
where $\Omega_{1}=\varphi_1(\mathbb D)$, $\Omega_{2}=\varphi_2(\mathbb D)$.
\end{thm}

A more detailed formulation is given in Section 4 (see Theorem \ref{thm:EstimateEigenBase}). In Section~5 we consider in more detail the case in which
$\Omega_1$ and $\Omega_2$ are quasidiscs.

The estimate for $|\lambda_{n}[\Omega_{1}]-\lambda_{n}[\Omega_{2}]|$ can also be given in terms of the measure variation:
\begin{multline}
|\lambda_{n}[\Omega_{1}]-\lambda_{n}[\Omega_{2}]|\\
\leq  c_n B_{p,\tau} \Big(\left[{\rm meas}\,(\varphi_1(\mathbb D^+))-{\rm meas}\,(\varphi_2(\mathbb D^+))\right]+
\left[{\rm meas}\,(\varphi_2(\mathbb D^-))-{\rm meas}\,(\varphi_1(\mathbb D^-))\right]\Big)^{\frac12},
\nonumber
\end{multline}
where
\begin{equation}\label{D+-}
\mathbb D^+=\{z\in\mathbb D:J_{\varphi_1}(z)\geq J_{\varphi_2}(z) \}\,, ~~~\mathbb D^-=\{z\in\mathbb D:J_{\varphi_1}(z)< J_{\varphi_2}(z) \}
\end{equation}
and $J_{\varphi_1}, J_{\varphi_2}$ are the Jacobians of the mappings $\varphi_1$, $\varphi_2$ respectively.

Inequalities (\ref{estimate 1}) and (\ref{estimate 2}) hold for any $\varphi_1,\varphi_2$ under consideration, but they are non-trivial only if
$$\|\varphi_1-\varphi_2\mid L^{1,\infty}(\mathbb D)\|< (\sqrt{c_n} A_\tau)^{-1}\,,  ~~~\|\varphi_1-\varphi_2\mid L^{1,2}(\mathbb D)\|< (\sqrt{c_n} B_{p,\tau})^{-1}$$
respectively, because the inequality $|\lambda_{n}[\Omega_{1}]-\lambda_{n}[\Omega_{2}]|<\sqrt{c_n}$ obviously holds for any $\lambda_{n}[\Omega_{1}], \lambda_{n}[\Omega_{2}]$.

In this article we adopt an investigation method based on the theory of composition operators \cite{U1,VU1}.
Let $\Omega\subset\mathbb{C}$ be an arbitrary bounded simply connected plane
domain. Consider the eigenvalue problem
for the Dirichlet Laplacian in $\Omega$
\[
\begin{cases}
-\Delta_{w}g(w)=\lambda g(w),\,\, w\in\Omega,\\
\,\,\,\,\, g\vert_{\partial\Omega}=0\,,
\end{cases}
\]
where $$
\Delta_w=\left(\frac{\partial^2}{\partial u^2}\right)+\left(\frac{\partial^2}{\partial v^2}\right),\,\,\,w=u+iv\,.
$$
By the Riemann Mapping Theorem there exists a conformal mapping
$\varphi:\mathbb{D}\to\Omega$ from the unit disc $\mathbb{D}$ to
$\Omega$. Then, by the chain rule for the function $f(z)=g\circ\varphi(z)$,
we have
\begin{multline*}
\Delta_{z}f(z)=\Delta_{z}(g\circ\varphi(z))=(\Delta_{w}g)(\varphi(z))\cdot|\varphi'(z)|^{2}\\
=-\lambda g(\varphi(z))\cdot|\varphi'(z)|^{2}=-\lambda|\varphi'(z)|^{2}f(z).
\end{multline*}
 Here $\Omega\ni w=\varphi(z),\,\,\, z\in\mathbb{D}$. Hence we obtain
the weighted eigenvalue problem for the Dirichlet Laplacian in the unit
disc $\mathbb{D}$
\[
\begin{cases}
-\Delta f(z)=\lambda h(z)f(z),\,\, z\in\mathbb{D}\,,\\
\,\,\,\,\, f\vert_{\partial\mathbb D}=0\,,
\end{cases}
\]
 where
\begin{equation}\label{h}
h(z):=|\varphi'(z)|^{2}=J_{\varphi}(z)=\frac {\lambda_{\mathbb{D}}^{2}(z)}{\lambda_{\Omega}^{2}(\varphi(z))}
\end{equation}
 is the hyperbolic (conformal) weight defined by the conformal mapping
$\varphi:\mathbb{D}\to\Omega$. Here
$\lambda_{\mathbb{D}}$ and $\lambda_{\Omega}$ are hyperbolic metrics
in $\mathbb{D}$ and $\Omega$ respectively \cite{BM}.

This means that the eigenvalue problem in $\Omega$ is equivalent to
the weighted eigenvalue problem in the unit disc $\mathbb{D}$.

In the sequel we consider the weak formulation the weighted eigenvalue problem, namely:
\begin{equation}
\iint\limits _{\mathbb{D}}(\nabla f(z)\cdot\nabla\overline{{g(z)}})~dxdy=\lambda\iint\limits _{\mathbb{D}}h(z)f(z)\overline{{g(z)}}~dxdy,\,\,~~\forall g\in W_{0}^{1,2}(\mathbb{D}).\label{WEn}
\end{equation}

The method suggested to study the weighted eigenvalue problem
for the Dirichlet Laplacian is based on the theory of composition operators
\cite{U1,VU1} and the ``\,transfer\,'' diagram
suggested in \cite{GGu}. Universal hyperbolic weights for weighted
Sobolev inequalities were introduced in \cite{GU4} (see also \cite{LM1}).

\section{The weighted eigenvalue problem}

Let $\Omega\subset\mathbb{C}$ be an open set on the complex plane.
The Sobolev space $W^{1,p}(\Omega)$, $1\leq p\le\infty$, is
the normed space of all locally integrable weakly differentiable functions
$f:\Omega\to\mathbb{R}$ with finite norm given by
\begin{align*}
\|f\mid W^{1,p}(\Omega)\|=\biggr(\iint\limits _{\Omega}|f(z)|^{p}\, dxdy\biggr)^{1/p}+\biggr(\iint\limits _{\Omega}|\nabla f(z)|^{p}\, dxdy\biggr)^{1/p}, \,\,1\leq p<\infty,
\\
\|f\mid W^{1,\infty}(\Omega)\|={\begin{array}{c}\\{\rm ess~sup}\\{z\in\Omega}\end{array}} |f(z)|+{\begin{array}{c}\\{\rm ess~sup}\\{z\in\Omega}\end{array}} |\nabla f(z)|.
\end{align*}

The Sobolev space $W_{0}^{1,p}(\Omega)$, $1\leq p\le\infty$, is the closure in the  $W^{1,p}(\Omega)$-norm of the space $C_{0}^{\infty}(\Omega)$ of all infinitely continuously differentiable functions with compact support in $\Omega$.

The seminormed Sobolev space $L^{1,p}(\Omega)$, $1\leq p\le\infty$, is the space of all locally integrable weakly differentiable functions
$f:\Omega\to\mathbb{R}$ with finite seminorm given by
\begin{align*}
\|f\mid L^{1,p}(\Omega)\|=\biggr(\iint\limits _{\Omega}|\nabla f(z)|^{p}\, dxdy\biggr)^{1/p}, \,\,1\leq p<\infty,
\\
\|f\mid L^{1,\infty}(\Omega)\|={\begin{array}{c}\\{\rm ess~sup}\\{z\in\Omega}\end{array}} |\nabla f(z)|.
\end{align*}

The weighted Lebesgue space $L^{p}(\Omega,h)$, $1\leq p<\infty$,
is the space of all locally integrable functions with the finite
norm
$$
\|f\mid L^{p}(\Omega,h)\|=\bigg(\iint\limits _{\Omega}|f(z)|^{p}h(z)~dxdy\biggr)^{\frac{1}{p}}.
$$
Here the weight $h:\Omega\to\mathbb{R}$ is a non-negative measurable function.

We define the weighted Sobolev space $W^{1,p}(\Omega,h,1)$, $1\leq p<\infty$,
as the normed space of all locally integrable weakly differentiable functions
$f:\Omega\to\mathbb{R}$ with the finite norm given by
$$
\|f\mid W^{1,p}(\Omega,h,1)\|=\|f\mid L^p(\Omega,h)\|+\|\nabla f\mid L^p(\Omega)\|.
$$

The following is an embedding theorem taken from \cite{GU4} and reformulated for the present situation.

\begin{thm} \label{thm:BoundEm} Let  $\Omega\subset\mathbb{C}$ be a bounded simply connected
domain and $\varphi:\mathbb{D}\to\Omega$
be a conformal mapping.

Then the weighted embedding operator
\begin{equation}\label{embedding}
i_{\mathbb{D}}:W_{0}^{1,2}(\mathbb{D})\hookrightarrow L^{2}(\mathbb{D},h)
\end{equation}
is compact and for any function $u\in W_{0}^{1,2}(\mathbb{D})$ the
inequality
\[
\|f\mid L^{2}(\mathbb{D},h)\|\leq K^{*}\|f\mid L^{1,2}(\mathbb{D})\|
\]
holds.

Here $h$ is the hyperbolic
(conformal) weight defined by equality $(\ref{h})$. The exact constant $K^{*}=1/\sqrt{\lambda_1[\Omega]}$, i.~e. is equal to the exact
constant in the inequality
$$\|g\mid L^{2}(\Omega)\|\leq K\|g\mid L^{1,2}(\Omega)\|\,,~~~~
\forall g\in W_{0}^{1,2}(\Omega)\,.$$
\end{thm}

\begin{proof} Since $\varphi^{-1}:\Omega\to\mathbb{D}$ is a conformal
mapping, the composition operator
\[
(\varphi^{-1})^{*}:L^{1,2}(\mathbb{D})\to L^{1,2}(\Omega),\,\,\,(\varphi^{-1})^{*}(f)=f\circ\varphi^{-1},
\]
 is an isometry \cite{GU4}. Let $f\in C_{0}^{\infty}(\mathbb{D})$,
then $g=(\varphi^{-1})^{*}(f)=f\circ\varphi^{-1}\in C_{0}^{\infty}(\Omega)$. 
So, for the function $g\in C_{0}^{\infty}(\Omega)$ the Poincar\'e inequality
\begin{equation}
\label{PoIn}
\|g\mid L^{2}(\Omega)\|\leq K^{*}\|g\mid L^{1,2}(\Omega)\|
\end{equation}
holds with the exact constant $K^{*}=1/\sqrt{\lambda_1[\Omega]}$. Hence, using
the ``\,transfer\,'' diagram \cite{GGu} we
obtain
\begin{multline*}
\|f\mid L^{2}(\mathbb{D},h)\|=\biggl(\iint\limits _{\mathbb{D}}|f(z)|^{2}h(z)~dxdy\biggr)^{\frac{1}{2}}
=\biggl(\iint\limits _{\mathbb{D}}|f(z)|^{2}J(z,\varphi)(z)~dxdy\biggr)^{\frac{1}{2}}\\
=\biggl(\iint\limits _{\Omega}|f\circ\varphi^{-1}(w)|^{2}~dudw\biggr)^{\frac{1}{2}}
\leq K^{*}\biggl(\iint\limits _{\Omega}|\nabla f\circ\varphi^{-1})(w))|^{2}~dudw\biggr)^{\frac{1}{2}}\\
=K^{*}\biggl(\iint\limits _{\mathbb{D}}|\nabla f(z))|^{2}~dxdy\biggr)^{\frac{1}{2}}=K^{*}\|f\mid L^{1,2}(\mathbb{D})\|.
\end{multline*}
 Approximating an arbitrary function $f\in W_{0}^{1,2}(\mathbb{D})$
by functions in the space $C_{0}^{\infty}(\Omega)$ we obtain that the inequality
\[
\|f\mid L^{2}(\mathbb{D},h)\|\leq K^{*}\|f\mid L^{1,2}(\mathbb{D})\|
\]
 holds for any function $f\in W_{0}^{1,2}(\mathbb{D})$.

Now we prove that the composition operator
\[
(\varphi^{-1})^{*}:W_{0}^{1,2}(\mathbb{D})\to W_{0}^{1,2}(\Omega)
\]
is bounded.

Let a function $f\in C_0^{\infty}(\mathbb D)$. The  composition $(\varphi^{-1})^{\ast}(f)=f\circ\varphi^{-1}$ belongs to $C_0^{\infty}(\Omega)$.
So, using the Poincar\'e inequality (\ref{PoIn}) and the boundedness of the composition operator
\[
(\varphi^{-1})^{*}:L^{1,2}(\mathbb{D})\to L^{1,2}(\Omega)\,,
\]
we have
\begin{multline*}
\|(\varphi^{-1})^{\ast}(f)\mid L^2(\Omega)\|\leq K^{*} \|\nabla\left((\varphi^{-1})^{\ast}(f)\right)\mid L^{1,2}(\Omega)\|\\
=K^{*}\|\nabla f \mid L^{1,2}(\mathbb D)\|\leq K^{*}\|f\mid W_0^{1,2}(\mathbb D)\|\,.
\end{multline*}
Here $K^{*}$ is the norm of the embedding operator $i:L^{1,2}(\Omega) \to L^2(\Omega)$, i.e the exact constant in the corresponding Poincar\'e inequality (\ref{PoIn}).

Therefore
\begin{multline*}
\|(\varphi^{-1})^{\ast}(f)\mid W_0^{1,2}(\Omega)\|=\|(\varphi^{-1})^{\ast}(f)\mid L^{2}(\Omega)\|+\|(\varphi^{-1})^{\ast}(f)\mid L^{1,2}(\Omega)\|\\
\leq  K^{*}\|\nabla f \mid L^{1,2}(\mathbb D)\|+\|\nabla f \mid L^{1,2}(\mathbb D)\|\leq ( K^{*}+1)\| f \mid W_0^{1,2}(\mathbb D)\|.
\end{multline*}

Approximating an arbitrary function $f\in W_{0}^{1,2}(\mathbb{D})$ by functions  in the space $C_{0}^{\infty}(\Omega)$
we obtain that the inequality
\[
\|(\varphi^{-1})^{\ast}(f)\mid W_0^{1,2}(\Omega)\|\leq ( K^{*}+1)\| f \mid W_0^{1,2}(\mathbb D)\|
\]
holds for any function $f\in W_{0}^{1,2}(\mathbb{D})$.

On the other hand
\begin{multline*}
\|f\mid L^{2}(\mathbb{D},h)\|=\biggl(\iint\limits _{\mathbb{D}}|f(z)|^{2}h(z)~dxdy\biggr)^{\frac{1}{2}}=\biggl(\iint\limits _{\mathbb{D}}|f(z)|^{2}J_{\varphi}(z)~dxdy\biggr)^{\frac{1}{2}}\\
=\biggl(\iint\limits _{\Omega}|f\circ\varphi^{-1}(w)|^{2}~dudv\biggr)^{\frac{1}{2}}=\|f\mid L^{2}(\Omega)\|
\end{multline*}
and the composition operator 
$$
\varphi^{*}: L^2(\Omega)\to L^{2}(\mathbb{D},h)
$$
is bounded ($\varphi^{*}(f)=f\circ\varphi$).

Hence the embedding
operator (\ref{embedding})
is compact as it is the composition of the bounded composition operator $\varphi^{*}:L^{2}(\Omega)\to L^{2}(\mathbb{D}, h)$
and  the compact embedding operator $i_{\Omega}:W_{0}^{1,2}(\Omega)\hookrightarrow L^{2}(\Omega)$.
\end{proof}

By Theorem \ref{thm:BoundEm} it immediately follows that the
spectrum of  the weighted eigenvalue problem (\ref{WEn})  with hyperbolic (conformal) weights $h$
is discrete and can be written in the form of a non-decreasing sequence
\[
0<\lambda_{1}[h]\leq\lambda_{2}[h]\leq...\leq\lambda_{n}[h]\leq...\,,
\]
where each eigenvalue is repeated as many times as its multiplicity. The weighted eigenvalue problem in the unit disc $\mathbb D$
is equivalent to the eigenvalue problem in the domain $\Omega=\varphi(\mathbb D)$ (see also, for example \cite{LM1}) and
\begin{equation}\label{equality}
\lambda_{n}[h]=\lambda_{n}[\Omega], \,\,\, n\in\mathbb N\,.
\end{equation}
For weighted eigenvalues (eigenvalues in $\Omega$) we have also the following properties:
\[
(i) \lim\limits _{n\to\infty}\lambda_{n}[h]=\infty\,,
\]

(ii) for each $n\in\mathbb{N}$
\begin{equation}
\lambda_n[\Omega]=\lambda_{n}[h]=\inf\limits _{\substack{L\subset W_{0}^{1,2}(\Omega)\\
\dim L=n
}
}\sup\limits _{\substack{f\in L\\
f\ne0
}
}\frac{\iint\limits _{\Omega}|\nabla f|^{2}~dxdy}{\iint\limits _{\Omega}|f|^{2}~dxdy}=
\inf\limits _{\substack{L\subset W_{0}^{1,2}(\mathbb D,h,1)\\
\dim L=n
}
}\sup\limits _{\substack{f\in L\\
f\ne0
}
}\frac{\iint\limits _{\mathbb D}|\nabla f|^{2}~dxdy}{\iint\limits _{\mathbb D}|f|^{2}h(z)~dxdy}\label{MinMax}
\end{equation}
 (Min-Max Principle), and
\begin{equation}
\lambda_{n}[h]=\sup\limits _{\substack{f\in M_{n}\\
f\ne0
}
}\frac{\iint\limits _{\mathbb D}|\nabla f|^{2}~dxdy}{\iint\limits _{\mathbb D}|f|^{2}h(z)~dxdy}\label{MaxPr}
\end{equation}
 where
\[
M_{n}={\rm span}\,\{\psi_{1}[h],...\psi_{n}[h]\}
\]
 and $\{\psi_{k}[h]\}_{k=1}^{\infty}$ is an orthonormal set of eigenfunctions
corresponding to the eigenvalues $\{\lambda_{k}[h]\}_{k=1}^{\infty}$.

(ii) If $n=1$, then formula (\ref{MinMax}) reduces to
\[
\lambda_{1}[\Omega]=\lambda_{1}[h]=\inf\limits _{\substack{f\in W_{0}^{1,2}(\Omega)\\
f\ne0
}
}\frac{\iint\limits _{\Omega}|\nabla f|^{2}~dxdy}{\iint\limits _{\Omega}|f|^{2}~dxdy}
=
\inf\limits _{\substack{f\in W_{0}^{1,2}(\mathbb D,h,1)\\
f\ne0
}
}\frac{\iint\limits _{\mathbb D}|\nabla f|^{2}~dxdy}{\iint\limits _{\mathbb D}|f|^{2}h(z)~dxdy}.
\]
 In other words
\begin{equation}
{\lambda_{1}[\Omega]}={\lambda_{1}[h]}=\frac{1}{\left(K^{\ast}\right)^2}\label{IneqC}
\end{equation}
 where $K^{\ast}$ is the sharp constant in the inequality
\begin{equation}
\left(\iint\limits _{\mathbb D}|f|^{2}h(z)~dxdy\right)^{\frac{1}{2}}
\leq K^{\ast}\left(\iint\limits _{\mathbb D}|\nabla f|^{2}~dxdy\right)^{\frac{1}{2}}\,,\,\,\,~~\forall f\in W_{0}^{1,2}(\Omega).\label{PoinIneq}
\end{equation}

\section{The $L^{1,2}$-seminorm estimates}

We consider two weighted eigenvalue problems in the unit disc $\mathbb D\subset\mathbb C$:
\[
\iint\limits _{\mathbb D}(\nabla f(z)\cdot\nabla\overline{{g(z)}})~dxdy=\lambda\iint\limits _{\mathbb D}h_{1}(z)f(z)\overline{g(z)}~dxdy\,,\,\,~~\forall g\in W_{0}^{1,2}(\mathbb D).
\]
 and
\[
\iint\limits _{\mathbb D}(\nabla f(z)\cdot\nabla\overline{{g(z)}})~dxdy=\lambda\iint\limits _{\mathbb D}h_{2}(z)f(z)\overline{{g(z)}}~dxdy\,,\,\,~~\forall g\in W_{0}^{1,2}(\mathbb D).
\]

The aim of this section is to estimate the ``\,distance\,''
between weighted eigenvalues $\lambda_{n}[h_{1}]$ and $\lambda_{n}[h_{2}]$.
\begin{lem}
 \label{lem:TwoWeight} Let $\mathbb D\subset\mathbb{C}$ be the unit disc
and let $h_{1}$, $h_{2}$ be conformal weights on $\mathbb D$. Suppose that there exists a constant $B>0$ such that
\begin{equation}
\iint\limits _{\mathbb D}|h_{1}(z)-h_{2}(z)||f|^{2}~dxdy\leq B\iint\limits _{\mathbb D}|\nabla f|^{2}~dxdy,\,\,
\forall f\in W_{0}^{1,2}(\mathbb D).\label{EqvWW}
\end{equation}

 Then for any $n\in\mathbb{N}$
\begin{equation}
|\lambda_{n}[h_{1}]-\lambda_{n}[h_{2}]|\leq \frac{B\tilde c_n}{1+B\sqrt{\tilde c_n}}< B\tilde c_n\,,\label{LemEq}
\end{equation}
where
\begin{equation}\label{tilde c_n}
\tilde c_n=\max\{\lambda_{n}^{2}[h_{1}],\lambda_{n}^{2}[h_{2}]\} \,.
\end{equation}
\end{lem}
\begin{proof} By (\ref{MaxPr})
\[
\lambda_{n}[h_{1}]=\sup\limits _{\substack{f\in M_{n}^{(1)}\\
f\ne0
}
}\frac{\iint\limits _{\mathbb D}|\nabla f|^{2}~dxdy}{\iint\limits _{\mathbb D}h_{1}(x)|f|^{2}~dxdy}\,,
\]
 where
\[
M_{n}^{(1)}={\rm span}\,\{\psi_{1}[h_{1}],...\psi_{n}[h_{1}]\}.
\]
 Hence, by (\ref{EqvWW}),
\[
\begin{split}
\lambda_{n}[h_{1}] & \geq\sup\limits _{\substack{f\in M_{n}^{(1)}\\
f\ne0
}
}\frac{\iint\limits _{\mathbb D}|\nabla f|^{2}~dxdy}{\iint\limits _{\mathbb D}h_{2}(z)|f|^{2}~dxdy+\iint\limits _{\mathbb D}|h_{1}(z)-h_{2}(z)||f|^{2}~dxdy} \\
 & \geq\sup\limits _{\substack{f\in M_{n}^{(1)}\\
f\ne0
}
}\frac{\iint\limits _{\mathbb D}|\nabla f|^{2}~dxdy}{\iint\limits _{\mathbb D}h_{2}(z)|f|^{2}~dxdy+B\iint\limits _{\mathbb D}|\nabla f|^{2}~dxdy}\\
& =\sup\limits _{\substack{f\in M_{n}^{(1)}\\
f\ne0
}
}\frac{\iint\limits _{\mathbb D}|\nabla f|^{2}~dxdy}{\iint\limits _{\mathbb D}h_{2}(z)|f|^{2}~dxdy}\cdot\frac{1}{1+B\frac{\iint\limits _{\mathbb D}|\nabla f|^{2}~dxdy}{\iint\limits _{\mathbb D}h_{2}(z)|f|^{2}~dxdy}}\\
& \geq\sup\limits _{\substack{f\in M_{n}^{(1)}\\
f\ne0
}
}\frac{\iint\limits _{\mathbb D}|\nabla f|^{2}~dxdy}{\iint\limits _{\mathbb D}h_{2}(z)|f|^{2}~dxdy}\cdot\inf\limits _{\substack{f\in M_{n}^{(1)}\\
f\ne0
}
}\frac{1}{1+B\frac{\iint\limits _{\mathbb D}|\nabla f|^{2}~dxdy}{\iint\limits _{\mathbb D}h_{2}(z)|f|^{2}~dxdy}}\\
& =\sup\limits _{\substack{f\in M_{n}^{(1)}\\
f\ne0
}
}\frac{\iint\limits _{\mathbb D}|\nabla f|^{2}~dxdy}{\iint\limits _{\mathbb D}h_{2}(z)|f|^{2}~dxdy}\cdot\frac{1}{1+B\sup\limits _{\substack{f\in M_{n}^{(1)}\\
f\ne0}}\frac{\iint\limits _{\mathbb D}|\nabla f|^{2}~dxdy}{\iint\limits _{\mathbb D}h_{2}(z)|f|^{2}~dxdy}}.
\end{split}
\]
 Since the function $F(t)={t}/{(1+Bt)}$ is non-decreasing on $[0,\infty)$
and by (\ref{MinMax})
\[
\sup\limits _{\substack{f\in M_{n}^{(1)}\\
f\ne0
}
}\frac{\iint\limits _{\mathbb D}|\nabla f|^{2}~dxdy}{\iint\limits _{\mathbb D}h_{2}(z)|f|^{2}~dxdy}\geq\lambda_{n}[h_{2}],
\]
 it follows that
\[
\lambda_{n}[h_{1}]\geq\frac{\lambda_{n}[h_{2}]}{1+B\lambda_{n}[h_{2}]}=\lambda_{n}[h_{2}]- \frac{B\lambda_{n}^2[h_{2}]}{1+B\lambda_{n}[h_{2}]}.
\]
 Hence
\begin{equation}
\lambda_{n}[h_{1}]-\lambda_{n}[h_{2}]\geq-\frac{B\lambda_{n}^2[h_{2}]}{1+B\lambda_{n}[h_{2}]}\geq-\frac{B\tilde c_n}{1+B\sqrt{\tilde c_n}}.\label{LemEqR}
\end{equation}
 For similar reasons
\[
\lambda_{n}[h_{2}]-\lambda_{n}[h_{1}]\geq -\frac{B\lambda^2_{n}[h_{1}]}{1+B\lambda_{n}[h_{1}]}
\]
 or
\begin{equation}
\lambda_{n}[h_{1}]-\lambda_{n}[h_{2}]\leq \frac{B\lambda^2_{n}[h_{1}]}{1+B\lambda_{n}[h_{1}]}\le \frac{B\tilde c_n}{1+B\sqrt{\tilde c_n}}.\label{LemEqL}
\end{equation}

Inequalities (\ref{LemEqR}) and (\ref{LemEqL}) imply inequality
(\ref{LemEq}). \end{proof}

\vskip 0.2cm \textbf{Remark.} By equality (\ref{IneqC})  the minimal value of $B$ in  inequality (\ref{EqvWW})
is equal to
\[
\frac{1}{\lambda_{1}[|h_{1}-h_{2}|]}\,.
\]
 Hence  inequality (\ref{LemEq}) implies that
\begin{equation*}
|\lambda_{n}[h_{1}]-\lambda_{n}[h_{2}]|\leq\frac{\max\{\lambda_{n}^{2}[h_{1}],\lambda_{n}^{2}[h_{2}]\}}{\lambda_{1}[|h_{1}-h_{2}|]}\,.\label{RemEqL}
\end{equation*}

Now we estimate the constant $B$ in Lemma \ref{lem:TwoWeight} in terms of ``\,distances\,'' between weights.

Recall that for any $2\le q<\infty$ the Sobolev inequality
\begin{equation}\label{Sobolev}
\|f\mid L^{q}(\mathbb{D})\|\leq C(q)\|\nabla f\mid L^{2}(\mathbb{D})\|
\end{equation}
 holds for any function $f\in W_{0}^{1,2}(\mathbb{D})$. We assume that  $C(q)$
is the best possible constant in this inequality.

\begin{lem}
\label{lem:TwoWeiPol} Let   $h_{1}$, $h_{2}$ be conformal weights on
$\mathbb{D}$ such that
\begin{equation}
d_{s}(h_{1},h_{2}):= \|h_1-h_2\mid L^{s}(\mathbb{D})\|<\infty  \label{EqvWWpol}
\end{equation}
 for some $1<s\le\infty$.

Then  inequality $(\ref{EqvWW})$ holds with the constant
\begin{equation}
B=\Big[C\Big(\frac{2s}{s-1}\Big)\Big]^2\,d_{s}(h_{1},h_{2})\,.\label{Lem2Es}
\end{equation}
\end{lem}

\begin{proof}
By the H\"older inequality and Sobolev inequality (\ref{Sobolev}) we get
\[
\iint\limits _{\mathbb{D}}|h_{1}(z)-h_{2}(z)||f|^{2}~dxdy
\]
\[
\leq\left(\iint\limits _{\mathbb{D}}\left(|h_{1}(z)-h_{2}(z)|\right)^{s}dxdy\right)^{\frac{1}{s}}\left(\iint\limits _{\mathbb{D}}|f(z)|^{\frac{2s}{s-1}}dxdy\right)^{\frac{s-1}s}
\]
\[
\leq \Big[C\Big(\frac{2s}{s-1}\Big)\Big]^2\,d_{s}(h_{1},h_{2}) \iint\limits _{\mathbb{D}}|\nabla f(z)|^{2}dxdy\,.
\]
\end{proof}

By the two previous lemmas we get immediately the main result for the difference of weighted eigenvalues:
\begin{thm}
\label{thm:TwoWW} Let  $h_{1}$, $h_{2}$ conformal weights on
$\mathbb{D}$. Suppose that $d_{s}(h_{1},h_{2})<\infty$ for some $s>1$.

Then, for every $n\in\mathbb{N}$,
\[
|\lambda_{n}[h_{1}]-\lambda_{n}[h_{2}]|\leq \tilde c_{n}\Big[C\Big(\frac{2s}{s-1}\Big)\Big]^2 d_{s}(h_{1},h_{2})\,.
\]
\end{thm}

\section{On ``distances" $d_{s}(h_{1},h_{2})$ for hyperbolic (conformal) weights $h_{1},h_{2}$}

Let us analyze ``\,distances\,'' $d_{s}(h_{1},h_{2})$ for hyperbolic (conformal) weights.

Recall that hyperbolic (conformal) weights $h_{1}(z),h_{2}(z)$ for bounded simply connected plane domains are Jacobians 
$J_{\varphi_{1}}(z)$, $J_{\varphi_{2}}(z)$ of conformal homeomorphisms
$$\varphi_{1}:\mathbb{D}\to\Omega_{1},\,\,\, \varphi_{1}:\mathbb{D}\to\Omega_{2}.
$$

Since $\Omega_{1},\Omega_{2}$ are
bounded domains the Jacobians $J_{\varphi_{1}}(z)$, $J_{\varphi_{2}}(z)$ are integrable,
i.~e. $\varphi_{1}',\varphi_{2}'\in L^{2}(\mathbb D)$. An example of the
unit disc without the interval $(0,1)$ on the horizontal axis demonstrates
that for general simply connected domains $\Omega$ the Jacobians of conformal homeomorphisms
$\varphi:\mathbb{D}\to\Omega$ need not be a power greater
than $1$. Hence the integrability of  Jacobians to the power $s>1$
is possible only under additional assumptions on $\Omega$.

 In \cite{GU5} it is proved that such integrability  is possible only for domains with finite
geodesic diameter. Hence  $d_{1}(h_{1},h_{2})<\infty$ but, for $s>1$, the quantity $d_{s}(h_{1},h_{2})$ is not defined
for all pairs of conformal weights $h_{1},h_{2}$.

\begin{lem}\label{Lemma 1}
\label{thm:MesureEstimate}Let $\varphi_{1}:\mathbb{D}\to\Omega_{1}$, $\varphi_{2}:\mathbb{D}\to\Omega_{2}$ be conformal homeomorphisms and
$h_{1},h_{2}$ be the corresponding conformal weights. Suppose that $\left|\varphi_{1}'\right|,\left|\varphi_{2}'\right|\in L^{p}(\mathbb{D})$ for some $2<p\le\infty$.

Then for $s=\frac{2p}{p+2}$
\begin{equation} \label{measure}
d_{s}(h_{1},h_{2})\leq \left(\|\varphi_1'\mid L^{p}(\mathbb D)\|+
\|\varphi_2'\mid L^{p}(\mathbb D)\|\right)
\cdot\||\varphi_1'|-|\varphi_2'|\mid L^{2}(\mathbb D)\|.
\end{equation}
\end{lem}

\begin{proof}
By the definitions of $h_1,h_2$ and $d_{s}(h_{1},h_{2})$
\begin{multline}
\left[d_{s}(h_{1},h_{2})\right]^{s}=\iint\limits _{\mathbb{D}}\left|h_{1}(z)-h_{2}(z)\right|^{s}dxdy
=\iint\limits _{\mathbb{D}}\left||\varphi'_{1}(z)|^{2}-|\varphi'_{2}(z)|^{2}\right|^{s}dxdy\\
=
\iint\limits_{\mathbb{D}}\left||\varphi'_{1}(z)|+|\varphi'_{2}(z)|\right|^{s}\left||\varphi_{1}'(z)|-|\varphi_{2}'(z)|\right|^{s}dxdy.
\nonumber
\end{multline}
Applying to  the last integral the H\"older inequality  with $r=\frac{2}{s}$
($1\le r<2$ because $1<s\le 2$) and $r'=\frac{r}{r-1}=\frac{2}{2-s}$
we obtain
\begin{multline}
\left[d_{s}(h_{1},h_{2})\right]^{s}\\
\leq\left(\iint\limits _{\mathbb{D}}\left||\varphi_{1}'(z)|+|\varphi_{2}'(z)|\right|^{\frac{2s}{2-s}}dxdy\right)^{\frac{2-s}{2}}
\left(\iint\limits _{\mathbb{D}}\left(|\varphi_{1}'(z)|-|\varphi_{2}'(z)|\right)^{2}dxdy\right)^{\frac{s}{2}}.
\nonumber
\end{multline}
Since $s=\frac{2p}{p+2}$ we have
\[
d_{s}(h_{1},h_{2})\leq \||\varphi_1'|+|\varphi_2'|\mid L^{p}(\mathbb D)\|\cdot\||\varphi_1'|-|\varphi_2'|\mid L^{2}(\mathbb D)\|\,.
\]
\end{proof}

Note that integral estimate (\ref{measure}) can be rewritten in terms of the measure variation.

\begin{lem}\label{Lemma 2}
Let $\varphi_{1}:\mathbb{D}\to\Omega_{1}$, $\varphi_{2}:\mathbb{D}\to\Omega_{2}$ be conformal homeomorphisms. Then
\begin{multline}
\||\varphi_1'|-|\varphi_2'|\mid L^{2}(\mathbb D)\|\\
\le
\Big(\left[{\rm meas}\,(\varphi_1(\mathbb D^+))-{\rm meas}\,(\varphi_2(\mathbb D^+))\right]+
\left[{\rm meas}\,(\varphi_2(\mathbb D^-))-{\rm meas}\,(\varphi_1(\mathbb D^-))\right]\Big)^{\frac12}\,,
\nonumber
\end{multline}
where the sets $\mathbb D^+$ and $\mathbb D^-$ are defined by equalities $(\ref{D+-})$.
\end{lem}

\begin{proof} By  using the elementary inequality $(a-b)^{2}\leq |a^{2}-b^{2}|$ for any
$a,b\ge0$ and the equality $|\varphi_{1}'(z)|^2=J_{\varphi}$ for conformal homeomorphisms we get
\begin{multline}
\iint\limits _{\mathbb{D}}\left(|\varphi_{1}'(z)|-|\varphi_{2}'(z)|\right)^{2}dxdy\\
\leq\iint\limits _{\mathbb{D}}\left|\left|\varphi_{1}'(z)\right|^{2}-\left|\varphi_{2}'(z)\right|^{2}\right|dxdy
=
\iint\limits _{\mathbb{D}}\left|J_{\varphi_{1}}(z)-J_{\varphi_{2}}(z)\right|dxdy\\
=\iint\limits _{\mathbb{D^+}}\left(J_{\varphi_{1}}(z)-J_{\varphi_{2}}(z)\right)dxdy
+\iint\limits _{\mathbb{D^-}}\left(J_{\varphi_{2}}(z)-J_{\varphi_{1}}(z)\right)dxdy\\
=\Big(\left[{\rm meas}\,(\varphi_1(\mathbb D^+))-{\rm meas}\,(\varphi_2(\mathbb D^+))\right]+
\left[{\rm meas}\,(\varphi_2(\mathbb D^-))-{\rm meas}\,(\varphi_1(\mathbb D^-))\right]\Big)\,.
\nonumber
\end{multline}
\end{proof}

By combining Lemma \ref{Lemma 1}, Theorem \ref{thm:TwoWW}, equality (\ref{equality}), by applying the triangle inequality and taking into account that
$\frac{2s}{s-1}=\frac{4p}{p-2}$ for $s=\frac{2p}{p+2}$, we obtain the main result of this paper:

\begin{thm}
\label{thm:EstimateEigenBase}
Let $\varphi_{1}:\mathbb{D}\to\Omega_{1}$, $\varphi_{2}:\mathbb{D}\to\Omega_{2}$ be conformal mappings.
Suppose that $\left|\varphi_{1}'\right|,\left|\varphi_{2}'\right|\in L^{p}(\mathbb{D})$ for some $2<p\le\infty$.

Then
for any $n\in\mathbb{N}$
\begin{equation*}\label{main estimate}
|\lambda_{n}[\Omega_{1}]-\lambda_{n}[\Omega_{2}]|\leq c_{n}\Big[C\Big(\frac{4p}{p-2}\Big)\Big]^2\Big(\|\varphi_1'\mid L^{p}(\mathbb D)\|+
\|\varphi_2'\mid L^{p}(\mathbb D)\|\Big)\||\varphi'_1|-|\varphi'_2|\mid L^{2}(\mathbb D)\|\,,
\end{equation*}
 where $\Omega_{1}=\varphi_1(\mathbb D)$, $\Omega_{2}=\varphi_2(\mathbb D)$ and $c_{n}$ is defined by equality $(\ref{c_n})$.
\end{thm}

By Lemmas \ref{Lemma 1} and \ref{Lemma 2} follows the estimate in terms of the measure variation:
\begin{multline}\label{main estimate 2}
|\lambda_{n}[\Omega_{1}]-\lambda_{n}[\Omega_{2}]|\leq c_{n}\Big[C\Big(\frac{4p}{p-2}\Big)\Big]^2\Big(\|\varphi_1'\mid L^{p}(\mathbb D)\|+
\|\varphi_2'\mid L^{p}(\mathbb D)\|\Big)\times\\
\times\Big(\left[{\rm meas}\,(\varphi_1(\mathbb D^+))-{\rm meas}\,(\varphi_2(\mathbb D^+))\right]+
\left[{\rm meas}\,(\varphi_2(\mathbb D^-))-{\rm meas}\,(\varphi_1(\mathbb D^-))\right]\Big)^{\frac12}\,.
\nonumber
\end{multline}

\section{Quasidiscs}

Now we describe a rather wide class of plane domains for which there exist conformal mappings with Jacobians of the class
$L^p(\mathbb D)$ for some $p>1$, i.e. with complex derivatives of the class
$L^p(\mathbb D)$ for some $p>2$.

\begin{defn}
A homeomorphism $\varphi:\Omega_1\to\Omega_2$
between planar domains is called $K$-quasiconformal if it preserves
orientation, belongs to the Sobolev class $W_{loc}^{1,2}(\Omega_1)$
and its directional derivatives $\partial_{\alpha}$ satisfy the distortion inequality
$$
\max\limits_{\alpha}|\partial_{\alpha}\varphi|\leq K\min\limits_{\alpha}|\partial_{\alpha}\varphi|\,\,\,
\text{a.e. in}\,\,\, \Omega_1\,.
$$
\end{defn}
Infinitesimally, quasiconformal homeomorphisms transform circles to ellipses
with eccentricity uniformly bounded by $K$. If $K=1$ we recover
conformal homeomorphisms, while for $K>1$ plane quasiconformal mappings need
not be smooth.
\begin{defn}
A domain $\Omega$ is called a $K$-quasidisc if it is the image of the
unit disc $\mathbb{D}$ under a $K$-quasiconformal homeomorphism of
the plane onto itself.
\end{defn}

It is well known that the boundary of any $K$-quasidisc $\Omega$
admits a $K^{2}$-quasi\-con\-for\-mal reflection \cite{GH1} and thus, for example,
any conformal homeomorphism $\varphi:\mathbb{D}\to\Omega$ can be
extended to a $K^{2}$-quasiconformal homeomorphism of the whole plane
to itself.

The boundaries of quasidiscs are called quasicircles. It is known that there are quasicircles for which no segment has finite length.
The Hausdorff dimension of quasicircles was first investigated by F. W. Gehring and J. V\"ais\"al\"a  \cite{GV73},
who proved that it can take all values in the interval $[1,2)$. S. Smirnov proved recently \cite{Smi10} that the Hausdorff dimension of
any $K$-quasicircle is at most $1+k^2$, where $k = (K-1)/(K +1)$.

Ahlfors's 3-point condition \cite{Ahl63} gives
a complete geometric characterization of quasicircles: a Jordan curve $\gamma$ in the plane is a quasicircle
if and only if for each two points $a, b$ in $\gamma$ the (smaller) arc between them has 
diameter comparable with $|a-b|$. This condition is easily checked for the snowflake.
On the other hand, every quasicircle can be obtained by an explicit snowflake-type
construction (see \cite{Roh01}).

For any planar $K$-quasiconformal homeomorphism $\varphi:\Omega_1\rightarrow\Omega_2$
the following sharp result is known: $J(z,\varphi)\in L_{loc}^{p}(\Omega_{1})$
for any $p<\frac{K}{K-1}$ (\cite{G1, A}).
\begin{prop}
\label{prop:confQuasidisc}Any conformal mapping $\varphi:\mathbb{D}\to\Omega$
of the unit disc $\mathbb{D}$ onto a $K$-quasidisc $\Omega$ belongs to $L^{1,p}(\mathbb{D})$
for any $1\le p<\frac{2K^{2}}{K^2-1}$.
\end{prop}

\begin{proof} Any conformal mapping $\varphi:\mathbb{D}\to\Omega$ can be extended to a $K^2$ quasiconformal homeomorphism  $\psi$
of the whole plane to the whole plane by reflection.
Since the domain $\Omega$ is bounded, $\psi$ belongs to the class $L^p(\Omega)$ for any $1\le p<\frac{2K^2}{K^2-1}$ (\cite{G1}, \cite{A}).
Therefore $\varphi$ belongs to the same class.
\end{proof}

Denote, for $K\geq 1$, by $A_K$ the class of all $K$-quasidiscs.
Theorem \ref{thm:EstimateEigenBase} and Proposition \ref{prop:confQuasidisc}  imply the following statement.

\begin{thm}
For any $K\geq 1$ there exists $p>2$ and $M>0$ such that, for any qiasidiscs $\Omega_1,\Omega_2\in A_K$ and conformal mappings 
$\varphi_k:\mathbb D\to\Omega_k$, k=1,2,
$|\varphi_1'|, |\varphi_2'|\in L^{p}(\mathbb D)$ and for any $n\in\mathbb{N}$
\begin{equation*}
|\lambda_{n}[\Omega_{1}]-\lambda_{n}[\Omega_{2}]|
\leq c_{n} M\Big(\|\varphi_1'\mid L^{p}(\mathbb D)\|+
\|\varphi_2'\mid L^{p}(\mathbb D)\|\Big)\,\|\varphi_1'-\varphi_2'\mid L^{2}(\mathbb D)\|\,,
\end{equation*}
 where $c_{n}$ is defined by equality $(\ref{c_n})$.
\end{thm}

\begin{proof}
Since $\frac{2K^2}{K^2-1}>2$, by Proposition \ref{prop:confQuasidisc} there exists $2<p<\frac{2K^2}{K^2-1}$,
say $p=\frac{2K^2-1}{K^2-1}$, such that $|\varphi'_1|,|\varphi'_2|\in L^p(\mathbb D)$. Therefore, by 
Theorem \ref{thm:EstimateEigenBase} the statement follows with, say $p=\frac{2K^2-1}{K^2-1}$ and 
$$
M=\left[C\left(\frac{4p}{p-2}\right)\right]^2=\left[C\left(4(2K^2-1)\right)\right]^2.
$$ 
\end{proof}

Acknowledgments.

The authors thank the anonymous reviewers for careful reading of the paper and really valuable comments.

\end{document}